\documentclass[11pt]{amsart}
\usepackage{mathrsfs,latexsym,amsfonts,amssymb}
\newtheorem{theorem}{Theorem}[section]
\newtheorem{lemma}[theorem]{Lemma}
\newtheorem{corollary}[theorem]{Corollary}
\newtheorem{question}[theorem]{Question}

\theoremstyle{definition}
\newtheorem{definition}[theorem]{Definition}

\theoremstyle{remark}

\begin{document}
	\title[The existence of suitable sets in locally compact strongly topological gyrogroups]
	{The existence of suitable sets in locally compact strongly topological gyrogroups}
	
	\author{Jiajia Yang}
	\address{(Jiajia Yang): School of mathematics and statistics, Minnan Normal University, Zhangzhou 363000, P. R. China}
	\email{yjj8030@163.com}
	
	\author{Jiamin He}
	\address{(Jiamin He): School of mathematics and statistics, Minnan Normal University, Zhangzhou 363000, P. R. China}
	\email{hjm1492539828@163.com}
	
	\author{Fucai Lin*}
	\address{(Fucai Lin): 1. School of mathematics and statistics, Minnan Normal University, Zhangzhou 363000, P. R. China; 2. Fujian Key Laboratory of Granular Computing and Application, Minnan Normal University, Zhangzhou 363000, P. R. China}
	\email{linfucai2008@aliyun.com; linfucai@mnnu.edu.cn}

	\thanks{The authors are supported by Fujian Provincial Natural Science Foundation of China (No: 2024J02022) and the NSFC (No. 11571158).}
\thanks{* Corresponding author}
	
	
\begin{abstract}
A subset $S$ of a topological gyrogroup $G$ is said to be a {\it suitable set} for $G$ if $S$ is discrete, the gyrogroup generated by $S$ is dense in $G$, and $S\cup \{0\}$ is closed in $G$, where $0$ is the identity element of $G$.	In this paper, it is proved that every locally compact strongly topological gyrogroup has a suitable set, which gives an affirmative answer to a question posed by F. Lin, et al. in \cite{key14}.		
\end{abstract}
	
	\maketitle	
\section{Introduction and Preliminaries}
In 1990, K.H. Hofmann and S.A. Morris in \cite{key12} introduced the concept of suitable set in topological groups, and they proved that each locally compact group has a suitable set. Comfort et al. in \cite{key6} and Dikranjan et al. in \cite{key7, key8} also give some classes of topological groups which have a suitable set, such as, each metrizable topological group has a suitable set.

The gyrogroup is a generalization of group and possess weaker algebraic structures, in which the associative law is not satisfied. Now the gyrogroups have been extensively studied in recent years. Indeed, in 1988, Ungar in \cite{key1988} first introduced the concept of gyrogroups when studying admissible velocities in Einstein velocity addition concerning $c$-ball. In 2017, W. Atiponrat in \cite{key2} gave the notion of topological gyrogroup, that is, a gyrogroup is endowed with a topology such that the multiplication and the inverse are continuous. Then, M. Bao and F. Lin in \cite{key3} studied some particular class of topological gyrogroups and introduced the concept of strongly topological gyrogroups; then they proved that every feathered strongly topological gyrogroup is paracompact, which gives a generalization a well-known theorem in topological groups. In 2020, F. Lin et al. in \cite{key14} consider suitable sets for (strongly) topological gyrogroups, and raised the following open question.

\begin{question}\cite[Question 4.17]{key14}\label{q0}
	Does each locally compact strongly topological gyrogroup have a suitable set?
\end{question}

In this paper, we mainly give an affirmative answer to Question~\ref{q0}.

Throughout this paper, if not specified, we assume that all topological spaces are Hausdorff. Moreover, the sets of the first infinite ordinal and positive integers are denoted by $\omega$ and $\mathbb{N}$ respectively. Readers may refer to \cite{key1, key9, key19} for terminology and notations not
explicitly given here. Next we recall some definitions.

\begin{definition}(\cite{key19}) Let $G$ be a nonempty set and $\oplus \colon G\times G\rightarrow G$ be a binary operation on $G$. Then the pair $(G,\oplus)$ is called a $groupoid$. A function $f$ from a groupoid $(G_{1} ,\oplus _{1})$ to a groupoid $(G_{2} ,\oplus _{2})$ is said to be a $groupoid \ homomorphism$ if $f(x_{1} \oplus _{1}x_{2})=f(x_{1})\oplus _{2} f(x_{2})$ for all $x_{1} ,x_{2} \in G_{1} $. Moreover, a bijective groupoid homomorphism from a groupoid $( G,\oplus)$ to itself will be called a $groupoid\ automorphism$. We mark the set of all automorphisms of a groupoid $(G,\oplus)$ as Aut $(G,\oplus)$.
\end{definition}
	
\begin{definition}(\cite{key19})
Let $(G,\oplus)$ be a nonempty groupoid. We say that $(G,\oplus)$ is a {\it gyrogroup} if the followings hold:
\\\indent(G1) There exists a unique identity element $0 \in G$ such that $0 \oplus x = x= x \oplus 0$ for every $x \in G;$
\\\indent(G2) For every $x \in G$, there exists a unique inverse element $\ominus x \in G$ such that $\ominus x\oplus x=0=x\oplus(\ominus x)$;
\\\indent(G3) For all $x, y \in G$, there exists a gyroautomorphism $gyr[x, y] \in Aut( G,\oplus)$ such that $x \oplus (y \oplus z) = (x \oplus y) \oplus gyr[x, y](z)$ for all $z \in G$;
\\\indent(G4) For any $x, y \in G$, $gyr[x \oplus y, y] = gyr[x, y].$
\end{definition}
	
\begin{definition}(\cite{key19})
Let $(G,\oplus)$ be a gyrogroup. A nonempty subset $H$ of $G$ is called a {\it subgyrogroup}, expressed as $H \leq G$, if the following statements hold:
\smallskip
\\\indent(1) The restriction $\oplus| _{H\times H}$ is a binary operation on $H$, i.e. $(H, \oplus| _{H\times H})$ is a groupoid;
\\\indent(2) For all $x, y \in H,$ the restriction of $gyr[x, y]$ to $H$ is a bijective homomorphism
from $H$ to itself;
\smallskip
\\\indent(3) $(H, \oplus| _{H\times H})$ is a gyrogroup.
\smallskip
\\\indent Furthermore, a subgyrogroup $H$ of $G$ is said to be an {\it $L$-subgyrogroup}, denoted by $H \leq _{L} G$, if $gyr[a, h](H) = H$ for all $a \in G$ and $h \in H$.
\end{definition}

\begin{definition}(\cite{key19})
The gyrogroup cooperation ``$\boxplus$'' is defined by the equation$$x\boxplus y=x\oplus gyr[x,\ominus y](y),\\\ x,y\in G.$$
\end{definition}	
	
\begin{definition}(\cite{key2})
A triple $(G, \tau, \oplus)$ is called a {\it topological gyrogroup} if and only if
\smallskip
\\\indent(1) $(G, \tau)$ is a topological space.
\smallskip
\\\indent(2) $(G, \oplus)$ is a gyrogroup.
\smallskip
\\\indent(3) The binary operation $\oplus \colon G\times G\rightarrow G$ is jointly continuous, where $G\times G$ is endowed with the product topology, and the operation of the inverse $\ominus :G\rightarrow G$, i.e. $x\rightarrow \ominus x$, is also continuous.
\end{definition}
	
\begin{definition}(\cite{key3}) Let $G$ be a topological gyrogroup. We say that $G$ is a {\it strongly \ topological\ gyrogroup} if there exists a neighborhood base $\mu$ of 0 such that for all $U \in \mu$, $gyr[x, y](U) = U$ for any $x, y \in G$. For convenience, we say that $G$ is a {\it strongly topological gyrogroup with neighborhood base $\mu$ of 0}.
\end{definition}	
 A well-known example of a strongly topological gyrogroup, which is not a topological group, is $M\ddot{o}bius$ topological gyrogroup, see \cite{key3}.

\begin{definition}(\cite{key14})
Let $G$ be a topological gyrogroup and $S$ is a subset of $G$. Then $S$ is said to be a {\it suitable set} for $G$ if $S$ is discrete, the gyrogroup generated by $S$ is dense in $G$, and $S\cup \{0\}$ is closed in $G$.
\end{definition}
	
\section{Locally compact strongly topological gyrogroup}
In this section, we mainly prove that every locally compact strongly topological gyrogroup has a suitable set, which gives an affirmative answer to Question~\ref{q0}. First, we give some technical lemmas.
	
\begin{lemma}\cite[Proposition 2.12]{key13}\label{l3}
Let  $(G, \tau, \oplus)$ be a topological gyrogroup and $H$ is a locally compact subgyrogroup of $G$, then $H$ is closed in $G$. 		
\end{lemma}
	
\begin{lemma}\label{l51}
Let $(G, \tau, \oplus)$ be a topological gyrogroup. If $U$ is an open neighborhood of 0 and $F$ is a compact subset of $G$, then there exists an open neighborhood $V$ of 0 such that $(a\oplus V)\oplus (b\oplus V)\subseteq (a\oplus b)\oplus U$ for every $a, b\in F$.
\end{lemma}
	
\begin{proof}
For each $a, b\in F$, since $G$ is a topological gyrogroup, there exists an open neighborhood $V_{a, b}$ of 0 such that $$\ominus((a\oplus V_{a, b})\oplus (b\oplus V_{a, b}))\oplus (((a\oplus V_{a, b})\oplus V_{a, b})\oplus ((b\oplus V_{a, b})\oplus V_{a, b}))\subset U.$$
Since $F\times F$ is a compact subset in $G\times G$ and $\{(a\oplus V_{a, b})\times (b\oplus V_{a, b}): (a, b)\in F\times F\}$ is an open cover of $F\times F$, there exists a finite subset $\{(a_{i}, b_{i}): i\leq n\}$ of $F\times F$ such that $F\times F\subset \bigcup_{i=1}^{n}((a_{i}\oplus V_{a_{i}, b_{i}})\times (b_{i}\oplus V_{a_{i}, b_{i}}))$. Now put $V=\bigcap_{i=1}^{n}V_{a_{i}, b_{i}}$. Then, for any $a, b\in F$, there exists $i\leq n$ such that $(a, b)\in (a_{i}\oplus V_{a_{i}, b_{i}})\times (b_{i}\oplus V_{a_{i}, b_{i}})$, which implies that $a\in a_{i}\oplus V_{a_{i}, b_{i}}$ and $b\in b_{i}\oplus V_{a_{i}, b_{i}}$, hence we have
\begin{align*}
\ominus(a\oplus b)\oplus((a\oplus V)\oplus (b\oplus V))&\subseteq \ominus((a_{i}\oplus V_{a_{i}, b_{i}})\oplus (b_{i}\oplus V_{a_{i}, b_{i}}))\oplus (((a_{i}\oplus V_{a_{i}, b_{i}})\oplus V)\\
&\oplus ((b_{i}\oplus V_{a_{i}, b_{i}})\oplus V))\\
&\subseteq\ominus((a_{i}\oplus V_{a_{i}, b_{i}})\oplus (b_{i}\oplus V_{a_{i}, b_{i}}))\oplus (((a_{i}\oplus V_{a_{i}, b_{i}})\oplus V_{a_{i}, b_{i}})\\
&\oplus ((b_{i}\oplus V_{a_{i}, b_{i}})\oplus V_{a_{i}, b_{i}}))\\
&\subseteq U.
\end{align*}
Therefore, we have $(a\oplus V)\oplus (b\oplus V)\subseteq (a\oplus b)\oplus U$ for every $a, b\in F$.
\end{proof}

The following lemma is due to A. Ungar, see \cite{key19}.
	
\begin{lemma}\label{l53}
Let $G$ be a gyrogroup. Then $x\boxplus (\ominus x)=0$ for each $x\in G$.
\end{lemma}
	
\begin{lemma}\label{l52}
Let $(G, \tau, \oplus)$ be a strongly topological gyrogroup with a symmetric neighborhood base $\mu$ at 0. If $U\in\mu$ and $H$ is a compact subset of $G$, then there exists $V\in \mu$ such that $(h\oplus V)\boxplus (\ominus h)\subseteq U$ for every $h\in H$.
\end{lemma}
	
\begin{proof}
For each $h\in H$, by Lemma~\ref{l53}, it follows from definition the operation `$\boxplus$' that there exists $V_{h}\in \mu$ such that $((h\oplus V_{h})\oplus V_{h})\boxplus (\ominus (h\oplus V_{h}))\subseteq U$. Since $H$ is compact and $H\subseteq \bigcup _{h\in H}(h\oplus V_{h})$, there is a finite subset $\{h_{1} , ..., h_{n}\} \subset H$ such that $H\subseteq \bigcup _{k=1}^{n}( h_{k}\oplus V_{h_{k}})$. Put $$V=\bigcap _{k=1}^{n}V_{h_{k} }.$$
Clearly, $V$ ia an open neighborhood of $0$ in $G$. For any $h\in H$, there exists $1\leq k\leq n$ such that $h\in h_{k}\oplus V_{h_{k}}$, hence
\begin{align*}
(h\oplus V)\boxplus (\ominus h)&\subseteq ((h_{k}\oplus V_{h_{k}})\oplus V)\boxplus (\ominus(h_{k}\oplus V_{h_{k}}))\\
&\subseteq((h_{k}\oplus V_{h_{k}})\oplus V_{h_{k}})\boxplus (\ominus(h_{k}\oplus V_{h_{k}}))\\
&\subseteq U.
\end{align*}
\end{proof}
	
\begin{lemma}\label{l5}
Let $(G, \tau, \oplus)$ be a strongly topological gyrogroup with a symmetric neighborhood base $\mu$ at $0$. If $U\in\mu$ and $H$ is a compact subset of $G$, then there exists $V\in \mu$ such that $(\ominus h)\oplus (V\oplus h)\subseteq U$ for every $h\in H$.
\end{lemma}
	
\begin{proof}
Clearly, we can choose $W\in\mu$ such that $W\oplus W\subseteq U$. For each $h\in H$, it follows that there exists $V_{h}\in \mu$ such that  $(\ominus h)\oplus (V_{h}\oplus h )\subseteq W$ and $\ominus (V_{h}\oplus h)\oplus (V_{h} \oplus h)\subseteq W$. Since $H$ is compact and $H\subseteq \bigcup _{h\in H}(V_{h}\oplus h) $, there is a finite subset $\{h_{1} , ..., h_{n}\} \subset H$ such that $H\subseteq \bigcup _{k=1}^{n}(V_{h_{k}}\oplus h_{k})$. Put $$V=\bigcap _{k=1}^{n}V_{h_{k}}.$$
Clearly, $V$ ia an open neighborhood of $0$ in $G$. For any $h\in H$, there exists $1\leq k\leq n$ such that $h\in V_{h_{k}} \oplus h_{k}$, hence
\begin{align*}
(\ominus h)\oplus (V\oplus h)&\subseteq \ominus( V_{h_{k} } \oplus h_{k})\oplus (V\oplus (V_{h_{k} }\oplus h_{k}))\\
&\subseteq\ominus (V_{h_{k} } \oplus h_{k} )\oplus (V\oplus (h_{k} \oplus W ))\\
&=\ominus (V_{h_{k} } \oplus h_{k} )\oplus((V\oplus h_{k})\oplus gyr[V, h_{k}](W))\\
&=\ominus (V_{h_{k} } \oplus h_{k} )\oplus((V\oplus h_{k})\oplus W)\\
&\subseteq\ominus (V_{h_{k} } \oplus h_{k} )\oplus((V_{h_{k} }\oplus h_{k})\oplus W)\\
&=(\ominus(V_{h_{k} } \oplus h_{k})\oplus(V_{h_{k} } \oplus h_{k}))\oplus W\\
&\subseteq W \oplus W\\
&\subseteq U.
\end{align*}
\end{proof}
	
A subgyrogroup $N$ of a gyrogroup $G$ is {\it normal} in $G$, denoted by $N \underline{\triangleleft}\ G$, if it is the kernel of a gyrogroup homomorphism of $G$.
We recall the following concept of the coset space of a topological gyrogroup. Clearly, each normal subgyrogroup is an $L$-subgyrogroup by \cite[Proposition 25]{key16}.

Let $(G, \tau, \oplus)$ be a topological gyrogroup and $N$ be a normal subgyrogroup of $G$. Then we can define a binary operation on the coset $G/N$ in the followings hold:$$(x\oplus N)\oplus (y\oplus N)=(x\oplus y)\oplus N,$$ for every $x,y \in G.$ By \cite[Theorem 27]{key16}, it follows that $G/N=\{x\oplus N:x\in G\}$ is a gyrogroup. We denote the mapping $\pi:G\rightarrow G/N, x\mapsto x\oplus N$. Clearly, $ \pi ^{-1} \{\pi(x)\}=x\oplus N $ for any $x\in G$. Denote by $\tau(G)$ the topology of $G$, the quotient topology on $G/N$ is as follows:$$\tau(G/N)=\{O\subseteq G/N:\pi^{-1}(O)\in \tau(G)\}.$$

The following lemma is important in the proof of the following Theorem~\ref{l6}.

\begin{lemma}\label{154}\cite[Theorem 31]{key20}
Let $H$ be a subgyrogroup of a gyrogroup $G$. Then $H\underline{\triangleleft}\ G$ if and only
if the operation on the coset space $G/H$ given by
$$(a\oplus H)\oplus (b\oplus H)=(a\oplus b)\oplus H$$ is well defined.
\end{lemma}

\begin{lemma}\cite[Theorem 3.8]{key3}\label{l1}
Assume that $(G, \tau, \oplus)$ is a topological gyrogroup and $N$ a compact normal subgyrogroup of $G$, then the quotient mapping of $G$ onto the quotient gyrogroup $G/N$ is perfect.
\end{lemma}
	
Similarly to the proof of \cite[Chapter II, Theorem 5.17]{key11} and \cite[Corollary 2.4.8]{key9}, we have the following lemma.
	
\begin{lemma}\label{l2}
Let $G$ and $H$ be topological gyrogroups and $\pi : G \to H$ be a continuous gyrogroup homomorphism. If $\pi$ is a quotient mapping, then $\pi$ is also an open mapping. In particular, if $\pi$ is a perfect mapping, then $\pi$ is an open mapping.
\end{lemma}
	
\begin{theorem}\label{l6}
Suppose that $(G, \tau, \oplus)$ is a $\sigma$-compact locally compact strongly topological gyrogroup. Then for every countable family $\{U_{n}: n\in \omega\}$ of neighborhoods of 0, there exists a family of symmetric open neighborhoods $\{V_{n}: n\in \omega \}$ of $0$ such that $\overline{V_{0}}$ is compact, $V_{n+1} \oplus V_{n+1}\subset V_{n}\cap U_{n}$, $gyr[x, y](V_{n})=V_{n}$ for each $n\in \omega$, $x, y\in G$ and the following statements hold:
\begin{enumerate}
\item $N=\bigcap _{n\in \omega }V_{n}$ is a compact normal subgyrogroup of $G$.
			
\item $x\oplus N=N\oplus x$ for every $x\in G$.
			
\item  $N$ is an $L$-subgyrogroup of $G$.
			
\item  $G/N$ is a strongly topological gyrogroup which has a countable base.
\end{enumerate}
\end{theorem}
	
\begin{proof}
Let $G$ be a strongly topological gyrogroup with a symmetric neighborhood base $\mu$ at 0 such that $gyr[x, y](W)=W$ for any $W\in\mu$ and $x, y\in G$. Suppose that $G=\bigcup _{n\in \omega }F_{n} $, where each $F_{n}$ is a compact subset of $G$, $0\in F_{n}$ and and $F_{n} \subset F_{n+1}, n\in \omega$. By Lemmas~\ref{l51},~\ref{l52} and~\ref{l5}, there exists a subfamily $\{V_{n}: n\in \omega\}\subset\mu$ such that $\overline{V_{0}}$ is compact and, for each $n\in \omega$, the following conditions hold:
		
\smallskip
(i) $V_{n+1} \oplus V_{n+1}\subset V_{n}\cap U_{n}$;
		
\smallskip
(ii) $gyr[x, y](V_{n})=V_{n}$ for any $x, y\in G$;
		
\smallskip
(iii) for any $x\in F_{n} $, $(\ominus x)\oplus (V_{n+1}\oplus x)\subseteq V_{n}$;
		
\smallskip
(iv) for any $x, y\in F_{n}$,  $(x\oplus V_{n+1})\oplus (y\oplus V_{n+1})\subseteq (x\oplus y)\oplus V_{n}$;
		
\smallskip
(v) for any $x\in F_{n}$,  $(x\oplus V_{n+1})\boxplus (\ominus x)\subseteq V_{n}$.
		
(1) Obviously, $\overline{V_{n+1}} \subset V_{n+1} \oplus V_{n+1}$ for each $n\in \omega$, hence $N=\bigcap _{n\in \omega }V_{n}=\bigcap _{n\in \omega }\overline{V_{n+1}} $ is closed in $G$. Therefore, $N\subseteq\overline{V_{0}} $ is compact in $G$. Moreover, it is obvious that $N$ is a subgyrogroup of $G$. Next we prove that $N$ is normal. By Lemma~\ref{154}, it suffices to show that $(x\oplus N)\oplus (y\oplus N)=(x\oplus y)\oplus N$ for any $x, y\in G$. Now fix any $x, y\in G$. Then $(x\oplus y)\oplus N\subset x\oplus (y\oplus gyr[y, x](N))=x\oplus (y\oplus N)$ by (ii), hence $(x\oplus y)\oplus N\subset(x\oplus N)\oplus (y\oplus N)$; moreover, by (iv), we have
\begin{align*}
(x\oplus N)\oplus (y\oplus N)&=(\bigcap_{n\in\omega}(x\oplus V_{n}))\oplus (\bigcap_{n\in\omega}(y\oplus V_{n}))\\
&=\bigcap_{n\in\omega}((x\oplus V_{n+1})\oplus (y\oplus V_{n+1}))\\
&\subset \bigcap_{n\in\omega}((x\oplus y)\oplus V_{n})\\
&=(x\oplus y)\oplus N.
\end{align*}
By the arbitrary choices of $x, y$ in $G$, the subgyrogroup $N$ is normal. Further, we can see that $G/N$ is a gyrogroup, which runs as in the proof of \cite[Theorem 27]{key16} or \cite[Theorem 29]{key20}.
Hence, $N=\bigcap _{n\in \omega }V_{n}$ is a compact normal subgyrogroup of $G$.
		
(2) By \cite[Theorem 32]{key20}, we have $x\oplus N=N\oplus x$ for any $x\in G$.
				
(3) By (1), since $N$ is normal, it follows that $N$ is an $L$-subgyrogroup of $G$.
		
(4) Let $\pi$ be the natural mapping of $G$ onto $G/N$. We claim that $G/N$ is a strongly topological gyrogroup. Indeed, by Lemma~\ref{l1} or \cite[Theorem 27]{key16}, it is easy to see that $G/N$ is a topological gyrogroup. Next we prove that the family $\{\pi(W): W\in\mu\}$ is a symmetric neighborhood base of the identity element $\widetilde{0}$ in $G/N$. Clearly, each $\pi(W)$ is symmetric and open by Lemma~\ref{l2}. Now take any $\widetilde{a}=a\oplus N, \widetilde{b}=b\oplus N\in G/N$ and $U\in\mu$. From \cite[Proposition 23]{key16}, it follows that
$$gyr[\widetilde{a}, \widetilde{b}](\pi(U))=\pi(gyr[a, b](U))=\pi(U).$$ Therefore, $G/N$ is a strongly topological gyrogroup.
				
We show that $\{\pi ( V_{n}) \colon n\in \omega \}$ is a countable base at $\widetilde{0}$ in  $G/N$. Assume that the family $\left\{w\oplus N\colon w\in W \right\}$ is an arbitrary neighborhood of $\widetilde{0}$ in  $G/N$, where $W$ is an open neighborhood of $0$ in $G$. Then there exists $n_{0} \in \omega$ such that $V_{n_{0} } \subset W\oplus N$. Otherwise, the family $\{\overline{V_{n}} \cap (W\oplus N)^{'}:n\in \omega\}$ of compact sets has the finite intersection property and thus $$\bigcap _{} \{\overline{V_{n}} \cap (W\oplus N)^{'}:n\in \omega\}\neq\emptyset.$$ This is impossible since $$\bigcap _{n\in \omega }(\overline{V_{n}} \cap (W\oplus N)^{'})=(\bigcap _{n\in \omega }\overline{V_{n}})\cap(W\oplus N)^{'}=(\bigcap _{n\in \omega }{V_{n}})\cap(W\oplus N)^{'}=N\cap((W\oplus N)^{'})=\emptyset.$$ Hence, $\pi(V_{n_{0} })\subset\{w\oplus N\colon w\in W\}.$ Since $G/N$ is homogeneous by \cite[Theorem 3.13]{key4}, it follows that $G/N$ is first-countable, hence it is metrizable by \cite[Theorem 2.3]{key5}. Moreover, since $G$ is $\sigma$-compact and $G/N$ is the continuous image of $G$, it follows that $G/N$ is $\sigma$-compact. Therefore, $G/N$ is separable since any $\sigma$-compact metrizable space is separable by \cite[Theorem 4.1.15]{key9}, hence it has a countable base.
\end{proof}
	
Let $D$ be an infinite set with the discrete topology and $a\not\in D$. Then $S(D)=D \cup\{a\}$ will denote the one-point compactification of $D$. Clearly, $\{S(D)\setminus F:|F|<\omega, F\subset D\}$ is a family of open neighborhoods at $a$. Therefore, $S(D)$ is a compact Hausdorff space of size $|D|$ having precisely one non-isolated point. The following lemma was given in \cite{key15}.
	
\begin{lemma}\cite[Fact 12]{key15}\label{l7}
Suppose that $X$ is a compact space with a single non-isolated point $x$, $Y$ is an infinite space and $f: X\rightarrow Y$ is a continuous. Then $Y$ is a compact space with a single non-isolated point $f(x)$.
\end{lemma}
	
Our next conclusion is the key point to establish suitable set in locally compact strongly topological gyrogroups.
	
\begin{lemma}\label{l8}
Let $G$ be a topological gyrogroup and $X$ be an infinite set with a discrete topology. If $f:S(X) \rightarrow G$ is a continuous map such that $f(a) =0$ and $\langle f(S(X))\rangle$ is dense in $G$. Then $ S=f(S(X))\setminus \{0\}$ is a suitable set for $G$.
\end{lemma}	
\begin{proof}
If $f(S(X))$ is a finite set, then $S$ is discrete. Since $S(X)$ is a compact space, it follows that $f(S(X))$ is a compact space. Hence, $S\cup\{0\}$ is closed. Because $\langle S \rangle =\langle S\cup\{0\}\rangle=\langle f(S(X))\rangle$ is dense in $G$, we conclude that $S$ is a suitable set for $G$.
		
Suppose now that $f(S(X))$ is an infinite set. By Lemma~\ref{l7}, the space $f(S(X))$ is a compact space with a single non-isolated point $f(a) =0$, where $a$ is the single non-isolated point of $S(X)$. Hence, $S$ is discrete and $S\cup\{0\}$ is compact and closed. By our assumption, the proof is completed.
\end{proof}
	
The proof of the following theorem is similar to \cite[Theorem 9]{key10}. Next, we give out the proof for the reader.
	
\begin{theorem}\label{l9}
If $(G, \tau, \oplus)$ is a compactly generated metrizable topological gyrogroup, then $G$ has a suitable set.
\end{theorem}
	
\begin{proof}
If $G$ is discrete, then it is obvious that $G$ has a suitable set. Now assume that $G$ is non-discrete. Let $\{ V_{n}:n\in \omega\}$ be a symmetric open neighborhood base at 0 such that $V_{0}=G$ and $V_{n+1} \subseteq V_{n}$ for any $n\in \omega$. Suppose that $G=\langle K \rangle $, where $K$ is a compact subset of $G$. Obviously, since $G$ is a compactly generated, it follows that $G$ is $\sigma$-compact; then $G$ is separable since any $\sigma$-compact metrizable space is separable by \cite[Theorem 4.1.15]{key9}. Let $D=\{ d_{n} :n\in \omega \}$ be a countable dense subset of $G$. For every $n\in \omega$, since $\{ x\oplus V_{n+1} :x\in G \}$ is an open cover of $G$ and $K$ is compact, there exists a finite subset $F_{n}$ of $G$ such that $K\subseteq \bigcup \{ x\oplus V_{n+1} :x\in F_{n} \}$, then $$G=\langle K \rangle \subseteq \langle \ \bigcup \{ x\oplus V_{n+1} :x\in F_{n} \} \rangle \subseteq \langle F_{n} \cup V_{n+1} \rangle. $$ Hence, $G=\langle F_{n} \cup V_{n+1} \rangle$.
		
By induction on $n\in\omega$ we will define a sequence $\{E_{n} :n\in \omega\}$ of finite subsets of $G$ with the following properties:
		
\smallskip		
(1)  $E_{n} \subseteq V_{n}$,
		
\smallskip		
(2)  $G\subseteq\langle E_{0} \cup E_{1} \cup \cdot \cdot \cdot \cup E_{n} \cup V_{n+1}\rangle$, and
		
\smallskip			
(3)  $d_{n}\in\langle E_{0} \cup E_{1} \cup \cdot \cdot \cdot \cup E_{n} \rangle$.
		
Set $E_{0} =F_{0} \cup \{d_{0} \}$. Obviously, $E_{0}$ satisfies (1)-(3). Assume that the finite sets $E_{0}, E_{1}, ... ,E_{n-1}$ have been defined satisfying the above properties (1)-(3). Clearly, $$F_{n} \cup \{d_{n} \}\subseteq \langle E_{0} \cup E_{1} \cup \cdot \cdot \cdot \cup E_{n-1} \cup V_{n} \rangle,$$ and since $F_{n}$ is finite, there exists a finite set $E_{n} \subseteq V_{n}$ such that $$F_{n} \cup \{d_{n} \}\subseteq \langle E_{0} \cup E_{1} \cup \cdot \cdot \cdot \cup E_{n-1} \cup E_{n} \rangle.$$ Hence, the above construction is completed.
		
From (1) it follows that the set $S=\bigcup\{E_{n} :n\in \omega\}$ forms a non-trivial sequence converging to 0. By (3), it follows that $D\subseteq \langle S \rangle$, and $\langle S \rangle$ is dense in $G$. Take any bijection $f:S(\mathbb{N})\rightarrow S$ and define also $f(a)=0$, $a \in S(\mathbb{N})$, then $\langle f(S(\mathbb{N})) \rangle=\langle S \rangle$ is dense in $G$. Therefore, by lemma~\ref{l8}, $G$ has a suitable set.
\end{proof}

{\bf Note:} Indeed, in \cite{key18}, we have proved that each metriable topological gyrogroup has a suitable set. However, the method of the proof is different.
	
\begin{theorem}\label{20}
Suppose that $G$ is a topological gyrogroup generated by its open subset with compact closure. Then $G$ has a suitable set.	
\end{theorem}

\begin{proof}
Let $O$ be an open subset in $G$ such that $\overline{U}$ is compact and $G=\langle \overline{U}\rangle$. Since $G$ is a homogeneous space, it follows that $G$ is locally compact. Moreover, $G$ a $\sigma$-compact because $G=\langle \overline{U}\rangle$. Therefore, $G$ satisfies the conditions of Theorem~\ref{l6}. Let $\{ U_{\alpha }:\alpha < \tau \}$ is a local base at 0 in $G$. If $\tau\leq \omega$, then the conclusion holds by Theorem~\ref{l9}. Now we can assume that $\tau>\omega$. Let $X$ be a subset of $G$ with $|X|=\tau$. For any $\alpha<\tau$, it follows from Theorem~\ref{l6} that there exists a compact normal $L$-subgyrogroup $N_{\alpha}$ of $G$ such that $N_{\alpha } \subseteq U_{\alpha }$ and $G/N_{\alpha}$ has a countable base. Let $\psi_{\alpha }:G\rightarrow G/N_{\alpha }$ be the quotient map. For any ordinal $\alpha$ satisfying $1\leq \alpha \leq \tau$ define $\varphi _{\alpha }=\Delta\ \{\psi _{\beta }:\beta <\alpha \} :G\rightarrow \prod \{G/N_{\beta}:\beta < \alpha \}$ and $G_{\alpha }=\varphi _{\alpha }(G)$. For $1\leq \beta \leq \alpha \leq \tau$ let $\rho _{\beta }^{\alpha }:\prod \{G/N_{\gamma }:\gamma < \alpha\}\rightarrow\prod \{G/N_{\gamma}:\gamma < \beta \}$ be the natural projection, and define $\pi _{\beta }^{\alpha } = \rho_{\alpha }^{\beta }\upharpoonright _{G_{\alpha }}:G_{\alpha }\rightarrow G_{\beta }$ to be the restriction of $\rho_{\alpha }^{\beta }$ to $G_{\alpha}\subseteq \prod \{G/N_{\gamma}:\gamma < \alpha\}$.
Next we will, by induction, to define a continuous map $f_{\alpha }:S(X)\rightarrow G_{\alpha }$ for each $\alpha$ satisfying $1\leq \alpha \leq \tau$ so that the following conditions hold:

\smallskip		
($\alpha 1$) $f_{\beta }=\pi _{\beta }^{\alpha }\circ f_{\alpha }$ for all $1\leq \beta < \alpha$;

\smallskip		
($\alpha 2$) $f_{\alpha }(a)=0_{G_{\alpha}}$;

\smallskip			
($\alpha 3$) $|\{x\in X:f_{\alpha}(x)\neq 0_{G_{\alpha}}\}|\leq \omega\cdot |\alpha |$;

 \smallskip			
($\alpha 4$) $\langle f_{\alpha}(S(X)) \rangle$ is dense in $G_{\alpha }$.

The above construction proof is similar to that of \cite[Theorem 18]{key15}.

By $(\tau 2)$ and $(\tau 4)$, we have $f_{\tau}(a)=0_{G_{\tau}}$ and $\langle f_{\tau}(S(X)) \rangle$ is dense in $G_{\tau}$. From Lemma~\ref{l8}, $S=f_{\tau}(S(X))\setminus\{0_{G_{\tau}}\}$ is a suitable set for $G_{\tau}$. Observe that $$\mbox{ker}\varphi _{\tau }\subseteq \bigcap \{N_{\alpha }:\alpha <\tau \}\subseteq \bigcap \{U_{\alpha }:\alpha < \tau \}=0,$$ then $\varphi _{\tau }:G\rightarrow G_{\tau}$ is an algebraic isomorphism. Hence $\varphi _{\tau }$ is a perfect map by \cite[Chapter II, Theorem 5.18]{key11} and \cite[Theorem 3.7.10]{key9}. Finally, note that each one-to-one continuous perfect map is a homeomorphism. Hence, $G$ and $G_{\tau}$ are isomorphic as topological gyrogroups. Therefore, $G$ has a suitable set.
\end{proof}
	
Next, let's give an affirmative answer to \cite[Question 4.17]{key14}.
	
\begin{theorem}\label{21}
If $(G, \tau, \oplus)$ is a locally compact strongly topological gyrogroup, then $G$ has a suitable set.	
\end{theorem}
	
\begin{proof}
Let $\mathcal{B}$ be a symmetric base at the identity element 0 such that $gyr[x, y](B)=B$ for any $x, y\in G$ and $B\in\mathcal{B}$. Since $G$ is locally compact, it can pick an open neighborhood $U\in \mathcal{B}$ at 0 such that $\overline{U}$ is a compact subset of $G$. Let $H$ be the subgyrgroup generated by $U$. Then it is easy to see that $H$ is an open $L$-subgyrogroup for $G$. In particular, $\overline{U} \subseteq \overline{H}=H$, and so $H$ is generated by its open subset with compact closure. By Theorem~\ref{20}, $H$ has a suitable set.
Hence, by \cite[Theorem 4.4]{key14}, we conclude that $G$ has a suitable set.
\end{proof}
	
\begin{corollary}\label{22}
Each compact strongly topological gyrogroup has a suitable set.
\end{corollary}

\begin{corollary}\cite[Theorem 1.12]{key12}
Each locally compact topological group has a suitable set.
\end{corollary}

{\bf Acknowledgements}. The authors wish to thank
the reviewers for careful reading preliminary version of this paper and providing many valuable suggestions.

\end{document}